\documentclass[12pt,amscd,amssymb]{amsart}																																					 

\usepackage[centertags]{amsmath}
\usepackage{amsmath}
\usepackage{mathrsfs}
\usepackage{latexsym}
\usepackage{amsthm}
\usepackage{amsfonts}
\usepackage{amssymb}
\usepackage{stmaryrd}

\usepackage{xspace}
\usepackage{verbatim}
\usepackage{enumerate}
\usepackage{indentfirst}

\theoremstyle{plain}
\DeclareMathOperator{\soc}{Soc}\DeclareMathOperator{\cdim}{cdim} \DeclareMathOperator{\alt}{Alt}
\DeclareMathOperator{\aut}{Aut}
\DeclareMathOperator{\defi}{d}

\newtheorem{theorem}{Theorem}[section]
\newtheorem{lemma}[theorem]{Lemma}
\newtheorem{fact}[theorem]{Fact}
\newtheorem*{theorem*}{Theorem}
\newtheorem*{maintheorem*}{Main Theorem}
\newtheorem*{lemma*}{Lemma}

\newtheorem{corollary}[theorem]{Corollary}
\newtheorem*{intconj*}{Intermediate Conjecture}
\newtheorem{prop}[theorem]{Proposition}
\newtheorem*{prop*}{Proposition}
\newtheorem*{conj*}{Principal Conjecture}
\newtheorem*{thm3.1*}{Theorem 3.1}
\newtheorem*{prop2.11*}{Proposition 2.11}

\theoremstyle{definition}
\newtheorem{definition}[theorem]{Definition}
\newtheorem{remark}[theorem]{Remark}
\newtheorem*{remark*}{Remark}
\newtheorem{case}{Case}

\newtheorem*{claim*}{Claim}

\newcommand{\los}{\L o\'{s}'s Theorem}
\begin{document}

\title[Pseudofinite groups in simple groups of fMr]{Pseudofinite groups as fixed points in simple groups of finite Morley rank}
\author[P. U\u{G}URLU]{P\i nar U\u{g}urlu$^{*,\dagger}$}
\thanks{$^*$Supported by T\"{U}B\.ITAK}
\thanks{$\dagger$Research partially supported by the Marie Curie Early Stage Training Network MATHLOGAPS (MEST-CT-2004- 504029)}
\address{Istanbul Bilgi  University \\
Yahya K\"{o}pr\"{u}s\"{u} Sok. No: 1
Dolapdere \\34440\\
Istanbul\\
Turkey}
 \email{pinar.ugurlu@bilgi.edu.tr}

\maketitle
\begin{abstract}
We prove that if the group of fixed points of a generic automorphism of a simple group of finite Morley rank is pseudofinite, then this group is an extension of a (twisted) Chevalley group over a pseudofinite field. On the way to obtain this result, we classify non-abelian definably simple pseudofinite groups of finite centralizer dimension (using the ideas of John S. Wilson \cite{W}).
\end{abstract}

\section*{Introduction}
A group of finite Morley rank is a group equipped with a rank function from the set of non-empty definable subsets to the non-negative integers. This rank function imitates the Zariski dimension in algebraic geometry (for details see \cite{bn}). Algebraic groups over algebraically closed fields are examples of groups of finite Morley rank in which case the Morley rank coincides with the Zariski dimension. Actually, the only known infinite simple groups of finite Morley rank are algebraic groups over algebraically closed fields. The \emph{Algebraicity Conjecture} (stated independently by Gregory Cherlin and Boris Zilber in the 1970's) says that any infinite simple group of finite Morley rank is an algebraic group over an algebraically closed field. Although this conjecture is still open, important results have been obtained by adapting and generalizing ideas from the classification of finite simple groups. This approach was suggested by Alexandre Borovik. For a detailed information about the Borovik program, we refer the reader to  \cite{abc}. The current status of the conjecture can be summarized as follows:

  In the theory of groups of finite Morley rank, the  Sylow $2$-subgroups are known to be conjugate and each of them  is a finite extension of the central product $U * T$ where $U$ is a definable connected $2$-group of bounded exponent and $T$ is a divisible abelian $2$-group.
Therefore, depending on the structure of the connected components of Sylow $2$-subgroups, the Algebraicity Conjecture breaks up into four cases.
\begin{enumerate}
	\item Even type: $U\neq 1$ and $T=1$  (identified with Chevalley groups over algebraically closed fields of characteristic $2$).
	\item Odd type: $U =1$ and $T\neq 1$ (structural results  are obtained on potential non-algebraic odd type groups).
	\item Mixed type: $U\neq 1$ and $T\neq 1$ (no such groups exist).
	\item Degenerate type: $U= 1$ and $T=1$ (difficult  case).
\end{enumerate}

A new approach to the Algebraicity Conjecture originates from results and ideas of Ehud Hrushovski in \cite{hru}, where some classes of structures, including simple groups of finite Morley rank, with generic automorphisms are considered. Hrushovski proved that the set of fixed points of a generic automorphism is a PAC structure with a definable measure (for details, see \cite{hru}). In the particular case of simple groups of finite Morley rank, the ultimate aim is to prove that the group of fixed points is pseudofinite, that is, an infinite model of the theory of finite groups.  However, unlike pseudofinite fields (which are characterized by James Ax in \cite{ax}), it is not known how to characterize pseudofinite groups.

 As it is mentioned by Hrushovski in \cite{hru}, the Algebraicity Conjecture implies the following conjecture.

 \begin{conj*}
Let $G$ be an infinite simple group of finite Morley rank with a generic automorphism $\alpha$. Then, the group of fixed points of $\alpha$  is pseudofinite.
\end{conj*}

This paper is a first step to construct a bridge between Algebraicity Conjecture and  Principle Conjecture from the other direction. More precisely, we aim to prove that the Principal Conjecture implies the Algebraicity Conjecture. This implication can be stated in the following form.
\begin{intconj*}
Let $G$ be an infinite simple group of finite Morley rank with a generic automorphism $\alpha$.  Assume that the group of fixed points of $\alpha$ is pseudofinite. Then, $G$ is isomorphic to a Chevalley group over an algebraically closed field.
\end{intconj*}

 In this paper,  we do not use the full strength of the genericity assumption on the automorphism $\alpha$. We can state the main result  as follows.
\begin{thm3.1*}
 Let $G$ be an infinite simple group of finite Morley rank and $\alpha$ be an automorphism of $G$ such that the definable hull of $C_G(\alpha)$ is $G$. If $C_G(\alpha)$ is pseudofinite, then there is a definable (in $C_G(\alpha)$) normal subgroup $S$ of $C_G(\alpha)$ such that  $S$ is isomorphic to a (twisted) Chevalley  group over a pseudofinite field and $C_G(\alpha)$ embeds in   $\aut(S)$.
\end{thm3.1*}

Here, $C_G(\alpha)$ denotes the group of fixed points of $\alpha$ in $G$. Note that the assumption on the definable hull of $C_G(\alpha)$ is satisfied by generic automorphisms of groups of finite Morley rank. Moreover, we will observe in the last section that under the assumptions of this theorem, degenerate types groups can not exist.

The first result we obtain on the way to prove Theorem~\ref{Theorem2} is the classification of non-abelian definably simple pseudofinite groups
of finite centralizer dimension.  For this classification, we first observe that  a sizeable part of  Wilson's classification proof for simple pseudofinite groups in \cite{W} works for non-abelian definably simple pseudofinite groups.  Then, with the help of  the finite centralizer dimension property, we show that such groups are elementarily equivalent to  (twisted) Chevalley groups over  pseudofinite fields.  Moreover, thanks to the results obtained by Mark Ryten  \cite{ryten}, we can replace the elementary equivalence by an isomorphism and we obtain the following result.
\begin{prop2.11*}  Every non-abelian definably simple  pseudofinite group of finite centralizer dimension is isomorphic to a (twisted) Chevalley  group over a pseudofinite field.
\end{prop2.11*}

Note that in some cases a non-abelian definably simple group $G$ is automatically  simple. For example, this holds if the theory of $G$ is supersimple \cite[Proposition 5.4.10]{wagner}. However, in our case we do not know whether the theory of a pseudofinite group of finite centralizer dimension is supersimple or not.

The structure of this paper can be outlined as follows.

In the first section, we give necessary background information about basic model theoretic concepts, ultraproducts, pseudofinite groups, groups of finite Morley rank, and we explain our terminology and notation.

   In the second section, firstly, we analyze the structure of abelian definably simple
pseudofinite groups. Then, we  classify non-abelian definably simple pseudofinite groups of finite centralizer dimension.

  In the third section, we prove the main result (Theorem~\ref{Theorem2}).

\section{Preliminaries}
This section covers some background material which will be necessary throughout this paper. We assume that the reader is familiar with the basic notions in model theory such as structure, language,  formula and theory.

Two structures $\mathfrak{M}$ and $\mathfrak{N}$, in a common language
$\mathcal{L}$, are \emph{elementarily equivalent} if they satisfy the
same $\mathcal{L}$-sentences and we write $\mathfrak{M}\equiv \mathfrak{N}$. A theory
is \emph{complete} if all of its models are elementarily equivalent.

A \emph{definable set} in a structure $\mathfrak{M}$ is a subset $X
\subseteq M^n$ ($M$ denotes the underlying set of the structure $\mathfrak{M}
$) which is the set of realizations of  a first order formula in the language of the structure $\mathfrak{M}$.

We emphasize that, throughout this paper, we stay within the borders of first
order logic and when we say definable, we mean definable possibly with parameters. We consider groups (resp. fields) as structures in the pure group (resp. field) language.

%This section covers some background material which will be necessary throughout this paper.

%In the first part, we  recall some basic model theoretic concepts very briefly.  \cite{changk}, \cite{hodges} and \cite{marker} are among the standard references for this topic.

%In the next part, we give some background information about ultraproducts and pseudofinite groups and list some related facts. Proofs of these facts can be found in the indicated references.  We refer the reader to  \cite{bells} and \cite{changk} for a detailed information about ultraproducts.

%In the last part, some important results in the theory of groups of finite Morley rank are listed. The proofs of these facts and  detailed information about groups of finite Morley rank can be found in  \cite{bn}. For detailed discussion of the classification project of simple groups of finite Morley rank, we refer the reader to  \cite{abc}.

\subsection{Ultraproducts and pseudofinite groups}

We give some background information about ultraproducts and pseudofinite groups and list some related facts. Proofs of these facts can be found in the indicated references.  We refer the reader to \cite{bells} and \cite{changk} for a detailed information about ultraproducts.

An \emph{ultrafilter} $\mathcal{U}$ on a non-empty set $I$ is a proper subset of the power set of $I$ such that $\mathcal{U}$ is  closed under finite intersections and taking supersets and  for any subset $A$ of $I$, we have $A\in \mathcal{U}$ if and only if $I\backslash A \notin \mathcal{U}$ . For any $i\in I$, the set $\left\{ X \subseteq I  ~|~ i \in X \right\}$ forms an ultrafilter which is called the \emph{principal ultrafilter generated by $i$} and any ultrafilter containing a finite set turns out to be a principal ultrafilter. Note that the existence of non-principal ultrafilters  is guaranteed by Zorn's lemma.

Now let $\prod_{i\in I} X_{i}$ denote the cartesian product of non-empty structures in a common language $\mathcal{L}$ and let $\mathcal{U}$ be an ultrafilter on $I$. Considering the elements of $\prod_{i\in I} X_{i}$ as functions from $I$ to $\bigcup_{i\in I} X_i$, define a relation on $\prod_{i\in I} X_{i}$ as follows. \[ x \sim_\mathcal{U} y \ \text{ if and only if }  \ \left\{  i \in I ~|~  x(i)= y(i) \right\} \in \mathcal{U}.\]
It is routine to check that $\sim_\mathcal{U}$ is an equivalence relation. The quotient of the cartesian product with respect to this equivalence relation is called the \emph{ultraproduct} and we denote it by $\prod_{i\in I} X_{i}/\mathcal{U}$. When all structures are the same, then this ultraproduct is called \emph{ultrapower}. We denote the equivalence class of  $x \in \prod_{i\in I} X_{i}$ with respect to $\sim_\mathcal{U}$ by  $[x]_\mathcal{U}$. An ultraproduct of $L$-structures has a natural (coordinatewise) $L$-structure.

The importance of the ultraproduct construction is expressed in \L o\'{s}'s Theorem \cite{los1} which states that a first order formula is satisfied in the ultraproduct if and only if it is satisfied in the structures indexed by a set belonging to the ultrafilter. In particular, first order properties of the ultraproduct  are determined by the first order properties of the structures in the ultraproduct together with the choice of the ultrafilter.

It can be observed that if $\mathcal{U}$ is an ultrafilter on a set $I$ which is a disjoint
union of finitely many subsets $I_1,\ldots, I_m$, then exactly one of $I_j$
is in $\mathcal{U}$. Moreover, whenever $J\in \mathcal{U}$ then, the set \ $\mathcal{U}_J = \left\{ X \cap J \ ~|~ \
X \in \mathcal{U} \right\}$ forms an ultrafilter on $J$ and we have
$\prod_{i \in I} X_i/\mathcal{U} \cong \prod_{j \in J} X_j/\mathcal{U}_J$. In particular, ultraproducts over principal
ultrafilters are isomorphic to one of the structures in the cartesian
product. Therefore, throughout this paper,  the ultrafilters in concern will always be non-principal.

\begin{remark}\label{wlog}
 Throughout the text,  we will use  the well-known properties mentioned in the previous paragraph. Moreover,  whenever we have $\prod_{i \in I} X_i/\mathcal{U} \cong \prod_{j \in J} X_j/\mathcal{U}_J$,  we will abuse the language, and keep writing $\mathcal{U}$ and $I$  for the ultrafilter and the index set.
\end{remark}

Ulrich Felgner  introduced in \cite{felgner} pseudofinite groups
as infinite models of the theory of finite groups in accordance with Ax's
characterization of pseudofinite fields \cite{ax}. By a suitable choice of an ultrafilter, it can be shown that any pseudofinite group is elementarily equivalent
to a non-principal ultraproduct of finite groups (see \cite{wilson}). As we will see in Section~\ref{simple}, the additive group of rational numbers, $(\mathbb{Q}, +)$, is a
pseudofinite group. However, the additive group of integers, $(\mathbb{Z}, +)$, is not a pseudofinite group, because the first order statement
$$\text{the map $x\mapsto x+x$ is one-to-one if and only if it is onto}$$
does not hold in $(\mathbb{Z},+)$ while it holds in every finite group.

Throughout this paper, a Chevalley group over an arbitrary field $K$  will be denoted by $X(K)$ where $X$ stands for one of the symbols $A_n (n \geqslant 1), B_n (n \geqslant 2), C_n (n \geqslant 3), D_n (n \geqslant 4), E_6, E_7, E_8, F_4, G_2$. These symbols come from the classification of finite dimensional complex simple Lie algebras. Chevalley constructed Chevalley groups over arbitrary fields associated to these symbols.  We will not go into details about the construction of neither these groups nor the twisted versions of these groups which can be constructed when the field $K$ has some additional properties (see \cite{s2}). We denote twisted Chevalley groups by $X(K)$ as well, where $X$ denotes one of the symbols  ${^2A_n} (n \geqslant 2), {^2D_n} (n \geqslant 4), {^3D_4}, {^2E_6}, {^2B_2}, {^2F_4}, {^2G_2}$.  It is known that (twisted) Chevalley  groups over arbitrary fields are simple as abstract groups except for ${^2A_2}(\mathbb F_4), {^2B_2}(\mathbb F_2), {^2F_4}(\mathbb F_2)$ and  ${^2G_2}(\mathbb F_3)$.  We refer the reader to \cite{carter} and \cite{s2} for details about these groups.

The following result of Fran\c{c}oise Point  shows that any infinite Chevalley group over an ultraproduct of finite fields is an example of a pseudofinite group.

\begin{fact}[Point \cite{P}] \label{Fact1} Let $\left\{ X(F_i)~|~ i\in I \right\}$ be a family of  (twisted) Chevalley groups of the same type $X$ over finite or pseudofinite fields, and let $\mathcal{U}$ be a non-principal ultrafilter on the set $I$. Then \[ \prod_{i \in I} X(F_i)/\mathcal{U} \cong X ( \prod_{i \in I} F_i/ \mathcal{U} ). \]
\end{fact}

If Fact~\ref{Fact1} is combined with Keisler-Shelah's Ultrapower Theorem \cite{keisler} then the following result can be obtained.
\begin{fact} \emph{(Wilson \cite{W})} \label{wilso}
Any group $G$ which is elementarily equivalent to a (twisted) Chevalley   group over a pseudofinite field is pseudofinite.
\end{fact}

\subsection{Groups of finite Morley rank}

We list  some basic properties of groups of finite Morley rank which will be used in the sequel without references. The proofs can be found in  \cite{bn} as well as in the indicated articles.

\begin{fact} [Macintyre \cite{mac2}]
A group of finite Morley rank has the descending chain condition on definable subgroups, that is, every descending chain of definable subgroups stabilizes after finitely many steps.
\end{fact}

The descending chain condition on definable subgroups is a strong property which allows one to define the notion of \emph{definable hull}. For any subset $X$ of a group of finite Morley rank $G$, there is a smallest definable subgroup of $G$ containing $X$, which is called the \emph{definable hull} of $X$ and denoted by $\defi(X)$.

%\begin{fact} [\cite{bn}, Lemma 5.34] \label{dc}

%\end{fact}

%Some properties of definable hull in a group of finite Morley rank $G$ are listed in the following fact.

%\begin{fact} [\cite{bn}, Lemma 5.35] \label{def} Let $G$ be a group of finite Morley rank, $X$ be a subset and $A$ be a subgroup of $G$.
%\begin{enumerate}

%	\item \label{124a} If elements of $X$ commute with each other, then $\defi(X)$ is abelian.
%	\item \label{124b} $C_G(X)= C_G(\defi(X))$.
%	\item \label{124c} If $A$ normalizes $X$, then $\defi(A)$ normalizes $\defi(X)$.
%\end{enumerate}
%\end{fact}

The following fact is a corollary of a result by John T. Baldwin and Jan Saxl \cite{bs}.
\begin{fact} [\cite{abc}, Corollary 2.9] \label{cdim}
For any subset $X$ of a finite Morley rank group $G$, the centralizer $C_G(X)$ is a definable subgroup. Moreover, there is an integer $n$ such that for any $Y \subseteq G$ there is  $Y_0 \subseteq Y$  with $|Y_0| \leqslant n$ and $C_G(Y) = C_G(Y_0)$.
\end{fact}

  A group with the property as in the moreover part of Fact~\ref{cdim} is said to have \emph{finite centralizer dimension}. More precisely, for any integer $k\geqslant 0$, a group has centralizer dimension $k$ if it has a proper descending centralizer chain of length $k$ and has no such chain of length greater than $k$. By a proper descending centralizer chain of length $k$ we mean the chain of the form
\[ G= C_G(x_0) > C_G(x_1) > C_G(x_1,x_2) >  \ldots > C_G(x_1,\ldots, x_k) = Z(G). \]
It is well-known that the class of groups with finite centralizer dimension is closed under taking subgroups and finite direct products \cite{myas}. Moreover, for any integer $k\geqslant 0$,  having centralizer dimension $k$ is a first order property in the language of groups (see \cite{duncan} and \cite{len}).

\section{Definably Simple Pseudofinite Groups} \label{simple}
The main result of this section is the characterization of non-abelian definably simple pseudofinite groups of finite centralizer dimension as (twisted) Chevalley  groups over pseudofinite fields. Firstly, we state the theorem of Wilson about the classification of simple pseudofinite groups and point out some results in the literature strengthening Wilson's theorem. Then, we  work on definably simple pseudofinite groups and after analyzing the structure of the abelian ones, we concentrate on the non-abelian case. We observe that the first part of Wilson's classification proof works for this case. Then, using the finite centralizer dimension property we obtain our result.

There is a weaker version of simplicity of a group which arises in model theory. A group is called \emph{definably simple} if it has no non-trivial definable proper normal
subgroups. In the class of non-abelian groups of finite Morley rank, which includes
non-abelian algebraic groups over algebraically closed fields,
definably simple groups coincide with the simple ones \cite{poizat}  (note that Wagner's result mentioned in the introduction is much more general).
However, in general, definably simple groups need not be simple. Wilson proved that no ultraproduct of alternating groups is simple (unless it is finite) while it is definably simple (see \cite{wilson}). Moreover, while no infinite abelian group is simple, some of them are
definably simple (see Fact~\ref{abelian}). Therefore, the distinction between the notions of \emph{definably simple} and
\emph{simple} becomes important.

Felgner conjectured in \cite{felgner} that simple pseudofinite groups are isomorphic to (twisted) Chevalley  groups over pseudofinite fields. Although Felgner obtained important results, it was Wilson who classified the simple pseudofinite
groups, however, only up to elementary equivalence.

\begin{fact} [Wilson \cite{W}]  \label{wilsontheorem} Every simple pseudofinite
group is elementarily equivalent to a (twisted) Chevalley  group
over a pseudofinite field.
\end{fact}
As it is pointed out by Wilson in \cite{W}, the elementary equivalence can be
replaced by an isomorphism for the untwisted case in view of results obtained by Simon Thomas in his dissertation \cite{T}. Moreover,  M. Ryten's recent results  in his thesis \cite{ryten} ensure that Wilson's theorem can be strengthened for both  Chevalley and twisted Chevalley groups and hence Felgner's conjecture is true.

Now, we concentrate on definably simple pseudofinite groups starting from the abelian ones. The following fact is a folklore.

\begin{fact}\label{abelian}
Let $A$ be an infinite abelian group. The following statements are equivalent.
\begin{enumerate}
\item $A$ is definably simple.
\item $A$ is torsion-free divisible.
\item $A\equiv \prod_{p \in I} C_{p} / \mathcal{U} \equiv ({\mathbb{Q}}, +)$ where  $\mathcal{U}$ is a non-principal ultrafilter on the set $I$ of all prime numbers and  $C_p$ is the cyclic group of order $p$.
\end{enumerate}
\end{fact}
\begin{proof}
Assume that $(a)$ holds. Then, any non-trivial element of finite order generates a non-trivial proper definable normal subgroup of $A$, so $A$ is torsion-free. Moreover,  if $A$ is not divisible, then $nA$ is a proper definable normal subgroup of $A$  for some integer $n\geqslant 2$ which is not possible by our assumption. Hence we get $(b)$. Assume that $(b)$ holds. Since the theory of divisible abelian  torsion-free groups is complete and all groups in part $(c)$ are models of this theory, we get the elementary equivalences in $(c)$.  As $\prod_{p \in I} C_{p}/\mathcal{U}$ is a definably simple group we have the implication $(c)\Rightarrow (a)$.
\end{proof}

The classification of  non-abelian definably simple pseudofinite groups will be obtained in two steps. Firstly, we show that every non-abelian definably simple pseudofinite group is elementarily equivalent to an
ultraproduct of non-abelian finite simple groups. This result follows from Wilson's proof of Fact~\ref{wilsontheorem}. However, for a  self-contained proof, we include here the results obtained by Wilson and we observe that his arguments work in the case of  non-abelian definably simple pseudofinite groups. In the second step, we proceed differently by using our
assumption on centralizer chains.

\begin{definition}\label{sigma}
Let $\sigma$ denote the following first order sentence  which was defined by Felgner in \cite{felgner}:
$$\forall x\forall y \left[ (x \neq 1 \wedge C_G(x, y)\neq 1) \rightarrow \bigcap_{g\in G}(C_G(x,y)C_G(C_G(x,y)))^g=1 \right]. $$
Above, $G$ stands for an arbitrary group, $C_G(x,y)$ denotes the centralizer of $\{x,y\}$ in $G$ and for any subset $X\subseteq G$, we have $X^g = gX g^{-1}$.
\end{definition}

Throughout this paper, $\soc(G)$ denotes the
subgroup generated by minimal normal subgroups of a group $G$, the so-called \emph{socle} of $G$.

Now, we recall some facts from \cite{W} which will be needed in the sequel.

\begin{fact}[Wilson \cite{W}] \label{Fact2} There is an integer $k$ such that
each element of each finite non-abelian simple group $G$ is a product of $k$
commutators.
\end{fact}

\begin{remark}\label{thompson}
A stronger version of this result is known as \emph{Ore's conjecture} which states that every element of a finite non-abelian simple group is a commutator \cite{O}. This old conjecture was followed by a stronger
conjecture, which is attributed to John Thompson,
stating that for every finite non-abelian  simple group $G$, there exists a conjugacy
class $C$ such that $G = CC$. We will refer to Thompson's conjecture later in the text. For a detailed information about the status of these conjectures see the survey article  \cite{kappe}. The proof of Ore's conjecture has been  recently completed (see \cite{ore}).
\end{remark}

\begin{fact} [Wilson \cite{W}] \label{Fact3}
Let $\sigma$ be the sentence from Definition \ref{sigma}.
\begin{enumerate}
\item If $G$ is a non-abelian simple group, then $G \models \sigma$, that is, $\sigma$ holds in $G$.

\item If $G$ is finite and $G \models \sigma $, then $\soc(G)$ is a
non-abelian simple group.
\end{enumerate}
\end{fact}

\begin{fact} [Wilson \cite{W}]
\label{Fact4} Let $G$ be a finite group with a
non-abelian simple socle. If $G$ is not simple then $G^{\prime }\neq G$.
Moreover, if every element of $\soc(G)$ is a product of $k$ commutators, then
every element of $G^{\prime }$ is a product of $k+3$ commutators.
\end{fact}

\begin{fact} [Wilson \cite{W}] \label{Fact5} Every simple pseudofinite group is elementarily equivalent to an ultraproduct of finite simple groups.
\end{fact}

In the following lemma, we observe that Wilson's proofs for Fact~\ref{Fact3}$(a)$  and Fact~\ref{Fact5} work for non-abelian definably simple pseudofinite groups.

\begin{lemma}
\label{Lemma1} If $G$ is a non-abelian definably simple pseudofinite group and $\sigma$ is the sentence from Definition \ref{sigma}, then the following statements hold.
\begin{enumerate}
\item $G \models \sigma$.
\item $G\equiv\prod_{i \in I} G_i / \mathcal{U}$ where each $G_i$ is a non-abelian finite simple group and $\mathcal{U}$ is a non-principal ultrafilter on a set $I$.
\end{enumerate}
\end{lemma}
\begin{proof}

$(a)$ Assume that $\sigma$ does not hold in $G$. Then, there are $x, y \in G
$ such that  $x \neq 1$ and $C_G(x,y) \neq 1$ and
\[
N := \bigcap_{g\in G}(C_G(x,y)C_G(C_G(x,y)))^g \neq 1.
\]
Clearly, $N$ is a definable normal subgroup of $G$.
Since $G$ is definably simple and $N$ is non-trivial by our assumption, we get $N=G$. As a result we have,
$C_G(x,y)C_G(C_G(x,y)) = G$. Therefore, $C_G(x,y)$ is normalized by $G$. Since $C_G(x,y)
$ is a non-trivial definable subgroup of $G$, we have $C_G(x,y)=G$. This is a contradiction, since
 $G$ can not have central elements as a non-abelian definably simple group.

$(b)$ As a non-abelian pseudofinite group,   $G$ is elementarily equivalent to $\prod_{i \in I} G_i/\mathcal{U}$ where each $G_i$ is a non-abelian
finite group and $\mathcal{U}$ is a non-principal ultrafilter on a set $I$.
Moreover, since $G\models \sigma$ by part~$(a)$, we may assume that $G_i\models \sigma$ for all $i\in I$ by \L o\'{s}'s Theorem and Remark~\ref{wlog}.   It follows by Fact~\ref{Fact3}$(b)$ that $\soc(G_i)$ is a
non-abelian simple group for each $i \in I$. Now, let $\varphi_i$ be a formula in the language of groups defining the set of products of $k+3$ commutators where $k$ is the integer given by Fact~\ref{Fact2}. By Fact~\ref{Fact4}, for any non-simple $G_i$, the formula $\varphi_i$ defines a proper normal subgroup. Therefore, if $G_i$ is not simple for almost all $i\in I$, then $G$ has a proper definable normal subgroup by \L o\'{s}'s Theorem. This is not possible as $G$ is definably simple and hence $G_i$ is a non-abelian
finite simple group for almost all $i\in I$. Again, by referring to Remark~\ref{wlog}, we can conclude that
$G\equiv \prod_{i \in I} G_i/\mathcal{U}$ where $G_i$ is a non-abelian finite simple group for all $i\in I$.
\end{proof}

\begin{corollary}
Every definably simple pseudofinite group is elementarily equivalent to an
ultraproduct of finite simple groups.
\end{corollary}
\begin{proof}
Follows from  Lemma~\ref{Lemma1}$(b)$ and Fact~\ref{abelian}.
\end{proof}

Now, we can prove our classification result.

\begin{prop} \label{Theorem1} Let $G$ be a non-abelian definably simple pseudofinite
group of finite centralizer dimension. Then $G$ is isomorphic to a (twisted) Chevalley
group over a pseudofinite field.
\end{prop}

\begin{proof}
By Lemma~\ref{Lemma1}$(b)$, we have $G\equiv \prod_{i \in I} G_i/\mathcal{U}$ where each $G_i$ is a non-abelian finite simple group. Since there are three families of non-abelian finite simple groups, without
loss of generality, we assume that just one family occurs in the
ultraproduct (see Remark~\ref{wlog}).  The possibilities are analyzed below.

\begin{case}
\emph{Sporadic Groups}

Since there are finitely many sporadic groups,
we may assume that all $G_i$'s are the same sporadic group $H$, that is,  $G \equiv H^I/\mathcal{U}$. However, this forces $G$ to be finite which is not the case.
\end{case}

\begin{case}
\emph{Alternating Groups}

It is well-known that the centralizer dimension of alternating groups increases as the rank increases. More precisely,  if we consider the centralizer chain of the form  $$C(1) > C((12)(34))> \cdots > C((12)(34), \ldots,
(k-3, k-2)(k-1, k)),$$ we can observe that the centralizer dimension of $\alt(n)$ is greater than  $\frac{n}{4}-1$. Therefore, the finite centralizer dimension property is satisfied only if there is a bound on the orders of the alternating groups in the ultraproduct. But then,  the ultraproduct is finite and this case is eliminated as well.
\end{case}

\begin{case}
\emph{Groups of Lie type}

In this case, all $G_i$'s are from one of the infinite families of classical
groups $A_n, B_n, C_n, D_n, {^2A_n}, {^2D_n}$. In the proof of \cite[Proposition 3.1]{tent}, it is shown that if there is no bound on the Lie ranks of the groups in the ultraproduct, then the centralizer dimension can not be finite.

%(!!!! ref tent argument proposition 3.1 )The structure of centralizers of semisimple elements of classical groups are
%well-known \cite{GLS3}. More precisely, in a classical group of type $X$,
%there are some special semisimple elements whose centralizers contain
%classical group of type $X$ of lower rank. In any classical group, this
%allows one to construct a descending chain of centralizers whose length
%increases with the rank. As a result, if there is no bound on the Lie ranks
%in the ultraproduct, finite centralizer dimension property is not satisfied and  this case is eliminated.
Therefore,  Lie ranks are bounded in the ultraproduct and we
may assume that all $G_i$'s are the groups of the same Lie type with the fixed Lie
rank $n$ and over fields of increasing order.
\end{case}
We conclude that $G$ is elementarily equivalent to a (twisted) Chevalley  group over a pseudofinite field by Fact~\ref{Fact1}. Moreover, we can strengthen the elementary equivalence to an isomorphism by the results obtained by  Ryten in \cite{ryten}.
\end{proof}
\begin{corollary}
Every non-abelian definably simple pseudofinite group of finite centralizer dimension is simple.
\end{corollary}

\section{Proof of the main result}

In this section, we will prove the following theorem.
\begin{theorem}\label{Theorem2}
 Let $G$ be an infinite simple group of finite Morley rank and $\alpha$ be an automorphism of $G$ such that the definable hull of $C_G(\alpha)$ is $G$. If $C_G(\alpha)$ is pseudofinite, then there is a definable (in $C_G(\alpha)$) normal subgroup $S$ of $C_G(\alpha)$ such that  $S$ is isomorphic to a (twisted) Chevalley  group over a pseudofinite field and $C_G(\alpha)$ embeds in   $\aut(S)$.
\end{theorem}

 Recall that, $C_G(\alpha )$ denotes the group of fixed points of $\alpha$ in $G$, that is, $C_G(\alpha) = \left\{g\in G ~|~ \alpha(g) = g \right\}$.  Moreover,  we would like to  introduce a simplified notation for  the  double centralizers since they will play an important role in our arguments. For any group $H$ and a subset $X$ of $H$, we will denote the the double centralizer $C_H(C_H(X))$ by $C^2_H(X)$ or even by $C^2(X)$ when the group in concern is specified clearly. Similarly the triple centralizer of $X$ in $H$ will be denoted by $C_H^3(X)$ or $C^3(X)$. It can be easily checked that $X\subseteq C_H^2(X)$ and  $C_H^3(X)=C_H(X)$.

We fix $G$ and $\alpha$ as in the Theorem~\ref{Theorem2} from now on.
\begin{remark}
It can be observed that simple groups of finite Morley rank of degenerate type, satisfying the assumptions of the Theorem~\ref{Theorem2}, do not exist:
It is known that simple groups of finite Morley rank of degenerate type do not have involutions \cite[IV, Theorem 4.1]{abc}. Therefore, $C_G(\alpha)$ has no involutions. Since this is a first order property, $C_G(\alpha)$ is an ultraproduct of finite groups of odd orders which are solvable by Thompson's odd order theorem. Evgenii I. Khukro proved in \cite{khukro} that such pseudofinite groups of finite centralizer dimension are solvable. However, then $G$ is solvable as the definable hull of $C_G(\alpha)$ \cite[I, Lemma 2.15]{abc}. This is not possible since $G$ is simple.
\end{remark}
%Assume that $G$ is a simple group of finite Morley rank of degenerate type, with an automorphism $\alpha$ such that $C_G(\alpha)$ is pseudofinite and $\defi(C_G(\alpha))=G$. It is known that simple groups of finite Morley rank of degenerate type do not have involutions (\cite{abc}!!!). Therefore, $C_G(\alpha)$ has no involutions. Since this is a first order property, $C_G(\alpha)$ is an ultraproduct of finite groups of odd orders which are soluble by Thompson's odd order theorem. E.I. Khukro proved in \cite{khukro} that such pseudofinite groups of finite centralizer dimension are solvable. However, then $G$ is solvable as the definable hull of $C_G(\alpha)$ (ref!!!). This is not possible since $G$ is simple.
%\end{proof}

%We need to prove several lemmas to obtain a proof of Theorem~\ref{Theorem2}.

\begin{lemma} \label{dense} Let $H$ be a  non-trivial subnormal subgroup of $C_G(\alpha)$. Then the following statements hold.
\begin{enumerate}
  \item  $\defi(H)=G$.
  \item $C_G(H) = 1$.
\end{enumerate}
\end{lemma}

\begin{proof}

$(a)$ Since $H$ is a subnormal subgroup  of $C_G(\alpha)$ there is a finite chain of subgroups of $C_G(\alpha)$ such that
 $$1\neq H \trianglelefteqslant H_1 \trianglelefteqslant \cdots \trianglelefteqslant H_{n-1}\trianglelefteqslant C_G(\alpha).$$
 Basic properties of definable hull (see \cite[Lemma 5.35]{bn}) ensure that if we take the definable hulls of each subgroup in the chain,  we get the following subnormal chain
 \[1\neq \defi(H) \trianglelefteqslant \defi(H_1) \trianglelefteqslant  \cdots \trianglelefteqslant  \defi(H_{n-1}) \trianglelefteqslant \defi(C_G(\alpha))=G.\]
 Since $G$ is simple, $\defi(H_{n-1})=G$ and inductively we get $\defi(H)=G$.

 $(b)$ Let $H$ be a non-trivial subnormal subgroup of $C_G(\alpha)$.  Since $H\leqslant C^2_G(H)$ and   $C^2_G(H)$ is a definable subgroup of $G$ (see Fact~\ref{cdim}), we have $\defi(H) \leqslant C^2_G(H)$. However, $\defi(H) = G$ by part (a) and we get $C_G(H) \leqslant Z(G)$. As a simple group, $G$ has trivial center and so,  $C_G(H)=1$.
\end{proof}

\begin{corollary}\label{nofinite}
There are no non-trivial subnormal subgroups of $C_G(\alpha)$ which are definable in $G$. In particular, $C_G(\alpha)$ has no non-trivial finite subnormal subgroups.
\end{corollary}

\begin{corollary}\label{cor2}
There are no non-trivial abelian subnormal subgroups of $C_G(\alpha)$.
\end{corollary}

\begin{lemma}\label{fop}
$C_G(\alpha)$ is elementarily equivalent to an ultraproduct of finite groups such that in each finite group, the centralizer of any non-trivial normal subgroup is trivial.
\end{lemma}
\begin{proof}
 As a pseudofinite group, \ $C_G(\alpha) \equiv \prod_{i\in I} G_i/\mathcal{U}$ \ where $G_i$ is a finite group for all $i\in I$ and $\mathcal{U}$ is a non-principal ultrafilter on a set $I$.  Let $\mu$ be the following sentence:
\[  \forall x \forall y \  \left[(x \neq 1 \wedge y \neq 1)  \rightarrow \exists z \ [y, x^z] \neq 1 \right].\] It can be observed that $\mu$ holds in a group $X$ if and only if the centralizer of any non-trivial normal subgroup of $X$ is trivial. By Lemma~\ref{dense}$(b)$,  $\mu$ holds in $C_G(\alpha)$. Therefore, we have $\prod_{i\in I} G_i/\mathcal{U} \models \mu$ and the lemma follows by \los \ and Remark~\ref{wlog}.
\end{proof}
\begin{remark}
From now on, we fix a family of finite groups \mbox{$\{ G_i ~|~ i\in I \}$} and an ultrafilter $\mathcal{U}$ on $I$ such that $C_G(\alpha) \equiv \prod_{i\in I} G_i/\mathcal{U}$ and the centralizer of any non-trivial normal subgroup of $G_i$ is trivial for each $i\in I$. Moreover, whenever we say ``for almost all $i\in I$'' we mean ``for elements of a subset of $I$ belonging to $\mathcal{U}$''.
\end{remark}
\begin{lemma}\label{one}  One of the following statements holds:
\begin{enumerate}
  \item $C_G(\alpha)$ is isomorphic to a (twisted) Chevalley group over a pseudofinite field.
  \item For each $i\in I$, there is a unique minimal normal subgroup $M_i$ of $G_i$ which is necessarily a direct product of non-abelian conjugate simple groups.
\end{enumerate}
\end{lemma}

\begin{proof}
$(a)$ If $G_i$ is simple for almost all $i\in I$, then $C_G(\alpha)$ is definably simple and hence isomorphic to a (twisted) Chevalley group over a pseudofinite field by Proposition~\ref{Theorem1}.

$(b)$ Assume that $G_i$ is not simple for almost all $i\in I$. If $G_i$ has two minimal normal subgroups $M_i$ and $N_i$, then  we have \ $M_i \cap N_i = 1$ \ and hence \  $[M_i , N_i] =1$. Therefore, $N_i$ centralizes $M_i$. However,  this contradicts to Lemma~\ref{fop} and hence, each $G_i$ has a unique non-abelian minimal normal subgroup $M_i$. Therefore, $\soc(G_i)=M_i$ for all $i\in I$.  Although it is well-known in finite group theory, we will recall the structure of such socles.  Let $S_{i}$ be a minimal normal subgroup of $M_i$. Since $M_i \trianglelefteqslant G_i$, the group $\left\langle S_{i}^x ~|~ x\in G_i \right\rangle$ is a normal subgroup of $G_i$ contained in $M_i$.  As $M_i$ is a minimal normal subgroup of $G_i$, we get $\left\langle S_{i}^x ~|~ x\in G_i \right\rangle = M_i$. Moreover,    $({S_i})^x$ is a  minimal normal subgroup of $M_i$ for each $x\in G_i$. Therefore,  the normal subgroup $({S_i})^x \cap \left\langle S_{i}^y ~|~ y\in G_i\setminus\{x\} \right\rangle$ of $M_i$ is either trivial or ${S_i}^x$ is contained in $\left\langle S_{i}^y ~|~ y\in G_i\setminus\{x\} \right\rangle$. In the latter case, we can eliminate $({S_i})^x$ from the product. As a result we can conclude that  \[\soc(G_i)=M_i = S_{i}^{g_{i0}} \times S_{i}^{g_{i1}} \times \ldots \times S_{i}^{g_{i{k_i}}},\] where $g_{ij}\in G_i$ for $0\leqslant j \leqslant k_i$. Note that if $S_i^{g_{ij}}$ has a proper non-trivial normal subgroup $N_i$ for some $0\leqslant j \leqslant k_i$, then $N_i$ is normalized by $M_i$. This is not possible as $S_i^{g_{ij}}$ is a minimal normal subgroup of $M_i$.
\end{proof}
In the following two lemmas, we observe that the structure of $\soc(G_i)$ can be simplified further.
\begin{definition}\label{generalization}
 For any group $H$ and integers $n\geqslant 1$, $k\geqslant 1$, let $\tau_{n,k}$ be the first order sentence which expresses the following statement. 
 \vspace{2mm}
For all  $\bar{x}_0, \bar{x}_1 \ldots, \bar{x}_n \in H^k\setminus\{(1,\ldots, 1)\}$ and for all $0 \leqslant i < j \leqslant n$, if  \[ [C_H^2(\bar{x}_i), C_H^2(\bar{x}_j)] = 1,\] then the product subgroup $C_H^2(\bar{x}_0)\ldots C_H^2(\bar{x}_n)$ is not normal in $H$.
\end{definition}

\begin{lemma}\label{sentence}
The sentence $\tau_{1,k}$  holds in $C_G(\alpha)$ for all $k\geqslant 1$.  	
\end{lemma}

\begin{proof}	
 Assume that  $\tau_{1,k}$ does not hold in $C_G(\alpha)$ for some $k\geqslant 1$. This means that, there are non-trivial $k$-tuples $\bar{x}, \bar{y}$ of elements of $C_G(\alpha)$ such that \[  [C^2(\bar{x}), C^2(\bar{y})] =1 \ \text{and}\ C^2(\bar{x})C^2(\bar{y}) \trianglelefteqslant C_G(\alpha), \]
where $C^2(\ldots)$ denotes the double centralizer in $C_G(\alpha)$. It is clear that $C^2(\bar{x})$ and $C^2(\bar{y})$ are non-trivial since they contain the tuples $\bar{x}$ and $\bar{y}$ respectively. %Moreover, they are  proper subgroups of $C_G(\alpha)$ because if $C^2(\bar{x})=C_G(\alpha)$ then  $C^2(\bar{y}) \leqslant Z(C_G(\alpha))$ which contradicts to Corollary~\ref{cor2}.
Therefore, $C^2(\bar{x})$ is a non-trivial subnormal subgroup of $C_G(\alpha)$ which is centralized by the non-trivial group $C^2(\bar{y})$. However, this contradicts to Lemma~\ref{dense}$(b)$.
\end{proof}

\begin{remark}\label{holds}
Note that the argument in the proof of Lemma~\ref{sentence} can be generalized to show that $\tau_{n,k}$ holds in $C_G(\alpha)$  for any $n \geqslant 1$, $k\geqslant 1$.
\end{remark}

\begin{lemma}\label{unique}
$\soc(G_i)$ is a non-abelian simple group for almost all $i\in I$.
\end{lemma}
\begin{proof} We continue using the same notation as in  Lemma~\ref{one}$(b)$, that is,  $S_i$ denotes a simple minimal normal subgroup of $\soc(G_i)$.  Let us work with a fixed $i\in I$ and consider the set $\left\{S_i^x ~|~ x\in G_i \right\}$. By eliminating the repeating conjugates of $S_i$ we get \[\left\{S_{i}^x ~|~ x\in G_i \right\}= \left\{S_i, S_i^{x_1}\ldots, S_i^{x_{m_i}}\right\}\] for some $x_1, \ldots, x_{m_i} \in G_i$. We will show that $S_i$ is the unique element of this set and hence $\soc(G_i)=S_i$.

 Note that all the centralizers mentioned below are taken in $G_i$ and $x_0=1$.
\vspace{2mm}

\noindent
{\bf Claim 1.} For any $0 \leqslant j < k \leqslant m_i$, we have $[C^2(S_i^{x_j}), C^2(S_i^{x_k})]=1$. Moreover, the product $C^2(S_i^{x_0})C^2(S_i^{x_1}) \ldots  C^2(S_i^{x_{m_i}})$ is a normal subgroup of $G_i$.
\vspace{2mm}
\noindent

The elements of any distinct pair of conjugates of $S_i$ commutes pairwise because $[S_i^{x_j}, S_i^{x_k}]\leqslant S_i^{x_j}\cap S_i^{x_k}=1$ by minimality of conjugates of $S_i$ in $M_i$.  Therefore, $S_i^{x_j}\leqslant C(S_i^{x_k})$ and so we have $C^2(S_i^{x_j}) \leqslant C^3(S_i^{x_k}).$ This means that $C^2(S_i^{x_j})$ centralizes $C^2(S_i^{x_k})$ and the first part of the claim follows. Since the components of the product commute pairwise and $G_i$ permutes them by conjugation, normality of the product follows.
\vspace{2mm}

\vspace{2mm}
Now, $\prod_{i\in I} G_i/\mathcal{U}$ has finite centralizer dimension, let us say $d$, since $C_G(\alpha)$ has this first order property. Therefore, without loss of generality, we may assume $\cdim(G_i)=d$ for all $i\in I$, by \los \ and Remark~\ref{wlog}.

\vspace{2mm}
\noindent
{\bf Claim 2.} For all $i\in I$, the group $C^2(S_i^{x_0})C^2(S_i^{x_1}) \ldots  C^2(S_i^{x_{m_i}})$ can be written as the product of at most $d$ factors.

\vspace{2mm}
\noindent

Let us fix $i\in I$. Since $S_i^{x_j}$ is a non-abelian group contained in  $C^2(S_i^{x_j})$, for each $0 \leqslant j \leqslant m_i$, we can choose an element  $c_{i_j}\in C^2(S_i^{x_j})\backslash  Z(C^2(S_i^{x_j}))$.  Let us consider the following chain of subgroups of $G_i$
 \[ C_{G_i}(c_{i_0}) \geqslant C_{G_i}(c_{i_0}, c_{i_1}) \geqslant \cdots \geqslant C_{G_i}(c_{i_0}, c_{i_1}, \ldots, c_{i_{m_i}}). \]  
  Now, if 
 \begin{equation*}
 C^2(S_i^{x_j})\nsubseteq C^2(S_i^{x_0})C^2(S_i^{x_1}) \cdots  C^2(S_i^{x_{j-1}}) \tag{$*$}
\end{equation*}
for some  $1 \leqslant j \leqslant m_i$, then
 $$C_{G_i}(c_{i_0}, c_{i_1}, \ldots, c_{i_{j-1}}) \supsetneqq C_{G_i}(c_{i_0}, c_{i_1}, \ldots, c_{i_{j-1}}, c_{i_j}).$$
This last inclusion is proper since $ C^2(S_i^{x_j}) \subseteq C_{G_i}(c_{i_0}, c_{i_1}, \ldots, c_{i_{j-1}})$ but there is an element in $C^2(S_i^{x_j})$ which does not commute with $c_{i_j}$. Since  the situation in $(*)$ can not occur more than $d$ times, the claim follows.

\vspace{2mm}
 By Claim~2, there are finitely many possibilities for the number  of factors of the product groups $C^2(S_i^{x_0})C^2(S_i^{x_1}) \ldots  C^2(S_i^{x_{m_i}})$. By Remark~\ref{wlog}, we can conclude without loss of generality that all of the product groups have the same number of factors, say $m+1$, where $m\geqslant 0$.

\vspace{2mm}
\noindent
{\bf Claim 3.} $m=0$. 
\vspace{2mm}
\noindent

Assume that $m\geqslant 1$ and $i\in I$. Since $\cdim(G_i)=d$, there are  $\bar{y}_0,\ldots, \bar{y}_{m}\in (G_i)^d$ such that $C^2(S_i^{x_j}) = C^2(\bar{y}_j)$ for all $0 \leqslant j \leqslant m$. By the previous claims and the definition of $m$, the sentence $\tau_{m,d}$ from Definition~\ref{generalization} does not hold in $G_i$ for all $i\in I$. Therefore, it does not hold in $C_G(\alpha)$ which contradicts to Remark~\ref{holds}.

\vspace{2mm}
Now, we are ready to show that $\soc(G_i)=S_i$ for all $i\in I$. Let us fix $i\in I$. By Claim~3, there is $x\in G_i$ such that for all $y\in G_i$ we have $C^2(S_i^{y}) \subseteq C^2(S_i^{x})$. If there is  $y\in G_i$ such that $S_i^x \neq S_i^y$, then by Claim~1, $C^2(S_i^{y})$ is abelian. This is not possible since  $C^2(S_i^{y})$ contains the non-abelian group $S_i^{y}$. Therefore, $S_i^x=S_i^y$ for all $y\in G_i$, that is $S_i \trianglelefteqslant G_i$. It follows that $\soc(G_i)=S_i$.
\end{proof}

In order to simplify the notation, any ultraproduct $\prod_{i\in I} X_i/\mathcal{U}$ will be denoted by $(X_i)_\mathcal{U}$ in the rest of the paper.

\begin{lemma}\label{definable}
The ultraproduct $(\soc(G_i))_{\mathcal{U}}$ is a definable normal subgroup of $(G_i)_{\mathcal{U}}$ which is isomorphic to a (twisted) Chevalley group over a pseudofinite field.
\end{lemma}

\begin{proof}
 Let us denote $\soc(G_i)$ by $M_i$ for each $i\in I$. By  Lemma~\ref{unique}, $M_i$ is a  non-abelian finite simple group for almost all $i\in I$. The ultraproduct $(M_i)_{\mathcal{U}}$ can not be finite because this would yield a finite normal subgroup of $C_G(\alpha)$ which is not possible by Corollary~\ref{nofinite}. Therefore, $(M_i)_{\mathcal{U}}$ is a pseudofinite group. Since $(M_i)_{\mathcal{U}}$ has finite centralizer dimension (inherited from $(G_i)_\mathcal{U}$), by Proposition~\ref{Theorem1}, we get that  $(M_i)_\mathcal{U}$ is isomorphic to an ultraproduct of  (twisted) Chevalley groups over finite fields.  In order to prove the definability of $(M_i)_{\mathcal{U}}$, we need  Thompson's Conjecture (see Remark~\ref{thompson}). Erich W. Ellers and Nikolai Gordeev  verified Thompson's Conjecture for finite simple groups of Lie type over  fields with more than $8$ elements \cite{EG}. Thompson's conjecture is applicable in our context since almost all of the fields in the ultraproduct have more than $8$ elements. It follows that, for almost all $i\in I$, there is  $r_i\in M_i$ such that $M_i = r_i^{M_i} r_i^{M_i}$.   Since $r_i\in M_i$ and $M_i \trianglelefteqslant G_i$, it is clear that  $r_i^{M_i}\subseteq r_i^{G_i}\subseteq M_i$. Therefore, $M_i=  r_i^{G_i} r_i^{G_i}$ and  $M_i$ is uniformly definable in $G_i$. It follows that $(M_i)_{\mathcal{U}}$ is definable in $(G_i)_\mathcal{U}$ as $(r_i)_\mathcal{U}^{(G_i)_\mathcal{U}} (r_i)_\mathcal{U}^{(G_i)_\mathcal{U}}$.
\end{proof}

\begin{proof}[Proof of Theorem~\ref{Theorem2}]
By Lemma~\ref{definable}, $(\soc(G_i))_\mathcal{U}$ is a definable normal subgroup of $(G_i)_\mathcal{U}$ which is isomorphic to a (twisted) Chevalley group over a pseudofinite field. Since $C_G(\alpha)$ is elementarily equivalent to $(G_i)_\mathcal{U}$, there is  a definable normal subgroup $S$ of $C_G(\alpha)$ which is elementarily equivalent to a (twisted) Chevalley group over a pseudofinite field. By \cite{ryten}, $S$ is isomorphic to a  (twisted) Chevalley group over a pseudofinite field. Moreover, since \[ S \trianglelefteqslant  C_G(\alpha) \ \ \text{and} \ \ C_{C_G(\alpha)}(S) =1,\] $C_G(\alpha)$ embeds in  $\aut(S)$ and the theorem follows.
\end{proof}
\vspace{2mm}

\noindent
{\it Acknowledgement}. \  This paper is a part of my Ph.D. thesis written under the supervision of Ay\c{s}e Berkman and Alexandre Borovik. I am very grateful
to my (co)supervisors for their suggestions and help. This work is supported by The Scientific and Technological Research Council of Turkey (T\"{U}B\.ITAK) and partially by MATHLOGAPS (MEST-CT-2004-504029).

\bibliographystyle{aabbrv}
\bibliography{pinar}

\end{document}